\documentclass{article}
\usepackage{authblk}
\usepackage[utf8]{inputenc}
\usepackage{amsfonts}
\usepackage{amssymb}
\usepackage{amsmath}
\usepackage{amsthm}
\usepackage{stmaryrd}
\usepackage{color}
\usepackage{enumerate, enumitem}

 \newtheorem{theorem}{Theorem}[section]
 \newtheorem{corollary}[theorem]{Corollary}
 \newtheorem{lemma}[theorem]{Lemma}

\newcommand{\Tr}{\mathrm{Tr}}
\newcommand{\R}{\mathbb{R}}
\newcommand{\C}{\mathbb{C}}
\newcommand{\State}{\mathcal{S}}
\newcommand{\Proj}{\mathcal{P}}

\newcommand{\Bsa}{\mathcal{B}_{sa}}
\newcommand{\Borel}{\mathfrak{B}}
\newcommand{\N}{\mathbb{N}}

\newcommand{\diam}{\mathrm{diam}}
\newcommand{\Lip}{\mathrm{Lip}}
\newcommand{\ran}{\mathrm{ran}}
\newcommand{\bbE}{\mathbb{E}}
\newcommand{\dm}{\,\,\mathrm{d}}
\newcommand{\eps}{\varepsilon}
\renewcommand{\Re}{\mathrm{Re}}

\newcommand{\px}{{|x\rangle\langle x|}}

\newcommand{\pure}[1]{\left| #1 \left\rangle\right\langle #1 \right|}
\newcommand{\varclass}[1]{\left\llbracket{#1}\right\rrbracket}

\newcommand{\jap}[1]{\langle#1\rangle}

\renewcommand{\[}{\begin{equation}}
\renewcommand{\]}{\end{equation}}

\title{When is the variance of one observable less than or equal to that of another with respect to all quantum states?}

\author[$\dagger$$+$]{Gy\"orgy P\'al Geh\'er}
\author[$*$]{Nazar Miheisi}

\affil[$\dagger$]{Department of Mathematics and Statistics, University of Reading, Whiteknights, P.O. Box 220, Reading RG6 6AX, United Kingdom, e-mail: gehergyuri@gmail.com}

\affil[$+$]{Riverlane, 59 St Andrew's St, Cambridge CB2 3DD, United Kingdom, e-mail: george.geher@riverlane.com}

\affil[$*$]{Department of Mathematics,
King's College London,
Strand, London WC2R 2LS,
United Kingdom,
email:~nazar.miheisi@kcl.ac.uk}

\date{}

\begin{document}

\maketitle

\begin{abstract}
    In quantum mechanics, the well-known Loewner order expresses that one observable's expectation value is less than or equal than that of
    another with respect to all quantum states. In this paper we propose and study a similar order relation in terms of the variance, and we prove two theorems. Our first result states that one observable's variance is less than or equal than that of another with respect to all quantum states if and only if the former is a $1$-Lipschitz function of the latter. The other main result we prove characterises the order automorphisms with respect to this proposed order relation. It turns out that in some sense these automorphisms have a more rigid form than in the case of the Loewner order.
\end{abstract}

\section{Introduction}

Throughout this paper $H$ denotes a complex Hilbert space. In the mathematical formulation of quantum mechanics, all quantum systems can be described using such a Hilbert space. More precisely, for every unit vector $x\in H$ the rank-one projection $\px$ represents a so-called \emph{pure state}, a general \emph{mixed state} is represented by a positive trace-class operator whose trace is $1$, and a (bounded) \emph{observable} corresponds to a (bounded) self-adjoint operator. Let us introduce the notation $\Proj_1(H) := \{\px\colon x\in H, \|x\|=1\}$ for the set of all pure states, $\State(H)$ for the set of all states, and $\Bsa(H)$ for the space of all bounded observables.

Let $\Proj(H)$ stand for the set of all (orthogonal) projections on $H$, and let the Borel $\sigma$-algebra of the real line $\R$ be denoted by $\Borel(\R)$.
Assume that our quantum system is in the state $\rho\in\State(H)$ and that we measure a physical quantity $A\in\Bsa(H)$ whose spectral measure is $E_A\colon \Borel(\R) \to \Proj(H)$. By the \emph{Born rule}, the following defines a Borel probability measure
\begin{equation}\label{eq:Born-rho}
    \mu_{A,\rho}\colon \Borel(\R)\to\R, \;\; \mu_{A,\rho}(\omega) = \Tr (E_A(\omega)\rho)
\end{equation}
such that $\mu_{A,\rho}(\omega)$ is the probability that our measurement produces a value from the set $\omega$. Note that both $E_A$ and $\mu_{A,\rho}$ are concentrated on the spectrum of $A$, which we will denote by $\sigma(A)$ from now on.
The \emph{expectation value} and the \emph{variance} of the observable $A$ in the quantum state $\rho$ are therefore
\begin{equation}\label{eq:exp-rho}
\mathbb{E}_\rho(A) := \Tr(\rho A)
\end{equation}
and 
\begin{equation}\label{eq:var-rho}
\Delta_\rho(A) := \mathbb{E}_\rho(A^2) - \mathbb{E}_\rho(A)^2 = \Tr(\rho A^2) - (\Tr(\rho A))^2,
\end{equation}
respectively. In case when $\rho=\pure{x}$ is a pure state, \eqref{eq:Born-rho}, \eqref{eq:exp-rho} and \eqref{eq:var-rho} become 
\begin{equation}\label{eq:Born}
    \mu_{A,x} := \mu_{A,\px} \colon \Borel(\R)\to\R, \;\; \mu_{A,x}(\omega) = \langle E_A(\omega)x,x\rangle,    
\end{equation}
\begin{equation}\label{eq:exp}
    \mathbb{E}_x(A) := \mathbb{E}_\px(A) = \langle Ax, x\rangle
\end{equation}
and 
\begin{equation}\label{eq:var}
    \Delta_x(A) := \Delta_\px(A) = \langle A^2x, x\rangle - \langle Ax, x\rangle^2.
\end{equation}

For two observables $A,B\in \Bsa(H)$ we write $A\leq B$ if and only if the expectation value of $A$ is less than or equal to that of $B$ with respect to all (pure) states $\rho$; that is,
$$\mathbb{E}_x(A) \leq \mathbb{E}_x(B) \quad (\px\in\Proj_1(H)),$$ 
or equivalently, 
$$\mathbb{E}_\rho(A) \leq \mathbb{E}_\rho(B) \quad (\rho\in\State(H)).$$ 
The partial order $\leq$ is usually called \emph{Loewner order} or positive semidefinite order in the literature. Due to its importance in quantum mechanics the Loewner order has been studied extensively, and many of its properties are quite well understood. In particular, the Loewner order automorphism group of $\Bsa(H)$ is known and is described by the following well-known theorem.

\begin{theorem}[Moln\'ar, \cite{Molnar-order1}]\label{thm:Molnar}
    Assume that $\phi\colon\Bsa(H)\to\Bsa(H)$ is a bijective map that preserves the Loewner order in both directions, that is,
    $$ A\leq B \;\;\iff\;\; \phi(A)\leq\phi(B) \quad (A,B\in\Bsa(H)).$$
    Then there exists a bounded bijective linear or conjugate-linear operator $T\colon H \to H$ and $S\in\Bsa(H)$ such that
    $$
    \phi(A) = TAT^* + S \quad (A\in\Bsa(H)).
    $$
\end{theorem}

We note that there has been a lot of interest in extending this theorem for positive operators and the effect algebra, and in generalising it to more abstract settings; see \cite{AS-order-abstract,GeSeCMP,IRS-order,Molnar-book,Molnar-order2,Mori-order1,Se-comp,Se-symm,Se-symm2,Se-order1,Se-order2,Se-order3,W-order} for more details.

\section{Statements of our main results}
The purpose of this paper is to propose and examine a similar order relation that is defined in terms of the variance (instead of the expectation value), and to prove Theorems \ref{thm:Lip1} and \ref{thm:pres}.

\subsection{$1$-Lipschitz functions and the variance order}
For two observables $A,B\in\Bsa(H)$ we write $A\preceq B$ if and only if the variance of $A$ is less than or equal to that of $B$ with respect to all pure states $\px$; that is,
\begin{equation}\label{eq:var-ord}
    \Delta_x(A) \leq \Delta_x(B) \quad (\px\in\Proj_1(H)).
\end{equation}
We propose to call $\preceq$ the \emph{variance order}.
As a consequence of our first main result (see Theorem \ref{thm:Lip1} below) we shall obtain that, as in the case of the expectation value, \eqref{eq:var-ord} is equivalent to
\begin{equation}\label{eq:var-ord2}
    \Delta_\rho(A) \leq \Delta_\rho(B) \quad (\rho\in\State(H)).
\end{equation}
%
%
Note that the variance order is reflexive and transitive. However, a short calculation gives that we have $A\preceq \epsilon A + cI$ and $A\succeq \epsilon A + cI$ for all $\epsilon\in\{1,-1\}$ and $c\in \R$, where $I$ denotes the identity operator, and therefore $\preceq$ on $\Bsa(H)$ is not antisymmetric. We shall fix this shortly by defining an equivalence relation on $\Bsa(H)$.

It is natural to ask whether the order relation $A\preceq B$ forces
$A$ and $B$ to be related in any other way? Our first goal in this paper
is to answer this question precisely. Surprisingly, it turns out that this relation implies a rather strong connection between the two observables, namely, that $A$ is a function of $B$. Recall that a function $f\colon K\to\R$ defined on a closed set $K\subset\R$ is \emph{Lipschitz} with Lipschitz constant $c>0$, or \emph{$c$-Lipschitz} for short, if we have 
$$
|f(x)-f(y)| \leq c|x-y| \quad (x,y\in K).
$$
Let us denote the set of all such functions by $\Lip_c(K)$.
It is well-known that such functions can be extended to be Lipschitz on the whole real line with the same constant $c$, see for instance Banach's book \cite{Lip-ext}. Now, we are in the position to state our first main result.

\begin{theorem}\label{thm:Lip1}
Assume that $H$ is a complex Hilbert space and $A, B\in\Bsa(H)$ are
bounded observables. Then the following conditions are equivalent:
\begin{enumerate}[label=(\roman*)]
    \item\label{thm:Lip1:pure} $\Delta_x(A) \leq \Delta_x(B)$ holds for all pure states $\px\in\Proj_1(H)$;
    \item\label{thm:Lip1:state} $\Delta_\rho(A) \leq \Delta_\rho(B)$ holds for all states $\rho\in\State(H)$;
    \item\label{thm:Lip1:function} there exists a Lipschitz function $f\colon \sigma(B)\to\R$ with Lipschitz constant $1$ such that $A = f(B)$.
\end{enumerate}
\end{theorem}
The implication \ref{thm:Lip1:state}$\implies$\ref{thm:Lip1:pure} is trivial
and so we will prove \ref{thm:Lip1:function}$\implies$\ref{thm:Lip1:state} and
\ref{thm:Lip1:pure}$\implies$\ref{thm:Lip1:function}; we do this in Section
\ref{sec:Lip1}.

Next, we say that two observables $A,B\in\Bsa(H)$ are in the same \emph{variance-equivalence class}, if we have $A\preceq B$ and $A\succeq B$. This indeed defines an equivalence relation. Let $\varclass{A} := \{B\in\Bsa(H)\colon A\preceq B, A\succeq B\}$ stand for the variance-equivalence class of $A$, and let us use the notation $\varclass{A}\preceq\varclass{B}$ if and only if $A\preceq B$. On the set of all variance-equivalence classes $\preceq$ clearly defines a partial order, as it is now also antisymmetric. The following is a straightforward consequence of Theorem \ref{thm:Lip1}.

\begin{corollary}
For all observables $A\in\Bsa(H)$ we have 
$$\varclass{A} = \{A+cI, -A+cI \colon c\in\R\}.$$
\end{corollary}

\subsection{Rigidity of variance order automorphisms}

The second main purpose of this paper is to characterise all variance order automorphisms of $\Bsa(H)$. Recall that every Loewner order automorphism is a composition of a conjugation by a bijective linear or conjugate-linear operator $T$ and a translation by an observable $S$ (see Theorem \ref{thm:Molnar}). 
Let us now give a few examples of variance order automorphisms. 
\begin{enumerate}[label=(\arabic*)]
    \item Every bijective map $\phi\colon\Bsa(H)\to\Bsa(H)$ that leaves each variance-equivalence class invariant (that is, $\phi(\varclass{A}) = \varclass{A}$ holds for all $A$) preserves the variances with respect to all states. Therefore these maps are clearly variance order automorphisms.
    \item For any positive number $\alpha$ the map $A\mapsto \alpha A$ is obviously a variance order automorphism, as $\Delta_\rho(\alpha A) = \alpha^2\Delta_\rho(A)$ holds for all $A$ and $\rho$.
    \item For each unitary or antiunitary operator $U\colon H\to H$, the map $A\mapsto UAU^*$ satisfies $\Delta_\rho(UAU^*) = \Delta_{U^*\rho U}(A)$ for all $A$ and $\rho$. Therefore these maps are also variance order automorphisms.
\end{enumerate}

%
It turns out that every variance order automorphism is a composition of these types of maps; see Theorem \ref{thm:pres} below. This is rather surprising, since in a sense there is more rigidity than for Loewner order automoprhisms. Of course, one should keep in mind that as the first example illustrates, a general variance order automorphism is not continuous, unlike in the case of the Loewner order (Theorem \ref{thm:Molnar}). On the other hand, if one concentrates only on the variance-equivalence classes, then a general variance order automorphism is simply a conjugation with a very special type of linear or conjugate-linear operator. Indeed note that unitary and antiunitary operators are exactly the bijective linear and conjugate-linear isometry operators respectively. This is why in some sense the variance order automorphisms have a more rigid form than Loewner order automorphisms. We state the related result below.


We will say that a topology $\mathcal{T}$ on $\Bsa(H)$ is a
\emph{unitary-antiunitarily invariant vector topology} if
$(\Bsa(H),\mathcal{T})$ is a (real and Hausdorff) topological vector space
such that the map $A\mapsto U^*AU$ is a homeomorphism of
$(\Bsa(H),\mathcal{T})$ for every unitary and antiunitary operator $U\colon H\to H$.

\begin{theorem}\label{thm:pres}
Suppose that $H$ is a separable complex Hilbert space of dimension at least $3$. Let $\Phi\colon \Bsa(H) \to \Bsa(H)$ be a bijective map that satisfies
\begin{equation}\label{eq:iff-preceq}
    B \preceq A \;\;\iff\;\; \Phi(B) \preceq \Phi(A) \quad (A,B\in \Bsa(H)).
\end{equation}
In case when $\dim H = \aleph_0$, we further assume that $\Phi$ is continuous
with respect to a unitary-antiunitarily invariant vector topology $\mathcal{T}$
in which the set $\{A\in\Bsa(H)\colon\#\sigma(A)<\infty\}$ is dense.
Then there exist an either unitary or antiunitary operator $U\colon H\to H$ and a positive number $\alpha$ such that
\begin{equation}\label{eq:pres-form}
    \Phi\left(\varclass{A}\right) = \varclass{\alpha UAU^*} \qquad (A\in\Bsa(H))
\end{equation}
is satisfied.
\end{theorem}

Note that the topology induced by the operator norm, and the weak/strong operator topologies
are unitary-antiunitarily invariant vector topologies
on $\Bsa(H)$ that satisfy the density requirement in the statement of the above theorem.
We shall prove this theorem in Section \ref{sec:pres}. Throughout the proof our previous result Theorem \ref{thm:Lip1} will play a crucial role. 

Before we proceed with the proofs, we would like to mention a result of Moln\'ar--Barczy, \cite[Theorem 2]{MoBa}, that characterises all bijective linear maps $\phi\colon\Bsa(H)\to\Bsa(H)$ preserving the maximal deviation (defined by $\|A\|_v := \sup\{\sqrt{\Delta_x(A)}\colon \px\in \Proj_1(H)\}$). Note that even though there are similarities between their and our results, there are some crucial differences too: we do not assume linearity of $\Phi$, and instead of preserving a quantity our map preserves only a relation. However, interestingly our result does not seem to imply the Moln\'ar-Barczy theorem.

\section{Proof of Theorem \ref{thm:Lip1}}\label{sec:Lip1}

\subsection{Proof of \ref{thm:Lip1:function}~$\implies$~\ref{thm:Lip1:state}}

Let $\mu$ be a Borel probability measure on $\R$. We will denote the
variance of $\mu$ by $Var(\mu)$. Then we have the standard identities
\[
Var(\mu) = \int_\R t^2 \dm\mu(t) - \left(\int_\R t \dm\mu(t)\right)^2
=\frac{1}{2}\int_\R\int_\R (t-s)^2 \dm\mu(t)\dm\mu(s).
\label{eq:variance}
\]
Let $S$ be the closed support of $\mu$. Then if $f:S\to\R$ is Borel
measurable, we will write $f_*\mu$ for the \emph{push-forward} of $\mu$
by $f$, i.e. $f_*\mu$ is the Borel measure on $\R$ defined by 
$$
(f_*\mu)(\omega) = \mu(f^{-1}(\omega)) \quad (\omega\in\mathfrak{B}(\R)).
$$
It is well-known that the push-forward measure satisfies
$$
\int_\R g \dm(f_*\mu) = \int_\R g\circ f \dm\mu
$$
for every continuous compactly supported function $g\colon\R\to\R$.

For an observable $A\in\Bsa(H)$ and state $\rho\in\State(H)$,
we will write $E_A$ for its (projection-valued) spectral measure and  $\mu_{A,\rho}$
for the probability measure defined in \eqref{eq:Born-rho}. Note that $\Delta_\rho(A)=Var(\mu_{A,\rho})$. It is well known that for every continuous function $f\colon\sigma(A)\to\R$ we have 
$$
E_{f(A)}(\omega) = E_A(f^{-1}(\omega)) \quad (\omega\in\mathfrak{B}(\R)),
$$
and therefore 
\[
\mu_{f(A),\rho}
= f_*\mu_{A,\rho}.
\label{eq:var-f(A)}
\]
The proof of the desired implication is now immediate.

\begin{proof}[Proof of Theorem \ref{thm:Lip1}
\ref{thm:Lip1:function} $\implies$ \ref{thm:Lip1:state}.]
Let $A,B\in\Bsa(H)$ be observables with $A=f(B)$ for some
$f\in\Lip_1(\sigma(B))$. Then using \eqref{eq:variance} and \eqref{eq:var-f(A)}
we see that for each state $\rho\in\State(H)$ we have
\begin{multline*}
\Delta_\rho(A) = \Delta_\rho(f(B))
=Var(f_*\mu_{B,\rho}) \\
= \frac{1}{2}\int_{\sigma(B)}\int_{\sigma(B)} (f(t)-f(s))^2
\dm\mu_{B,\rho}(t)\dm\mu_{B,\rho}(s).
\end{multline*}
Since $|f(t)-f(s)|\leq|t-s|$, we conclude that
$$
\Delta_\rho(A)
\leq
\frac{1}{2}\int_{\sigma(B)}\int_{\sigma(B)} (t-s)^2\dm\mu_{B,\rho}(t)\dm\mu_{B,\rho}(s)
= Var(\mu_{B,\rho})
= \Delta_{\rho}(B).
$$
\end{proof}

\subsection{Proof of \ref{thm:Lip1:pure}~$\implies$~\ref{thm:Lip1:function}
when $\dim H<\infty$}

In order to elucidate some of the main ideas in proving the implication
\ref{thm:Lip1:pure}~$\implies$~\ref{thm:Lip1:function}, we first give a
short proof in the simpler case where $H$ is finite dimensional. We will
make use of the fact that in this setting, each observable has a basis of
eigenvectors.

An important observation (here, and in the infinite dimensional setting) is
that for an observable $A\in\Bsa(H)$ and a unit vector $x\in H$, $\Delta_x(A)$
is a quantitative measure of how close $x$ is to being an eigenvector of $A$. Indeed,
\begin{align*}
\Delta_x(A) &= \langle A^2x, x\rangle - \langle Ax, x\rangle^2 \\
&= \langle Ax, Ax - \bbE_x(A)x\rangle \\
&= \|Ax - \bbE_x(A)x\|^2 + \bbE_x(A)\langle x, Ax - \bbE_x(A)x \rangle.
\end{align*}
But $\langle x, Ax - \bbE_x(A)x \rangle = \langle x, Ax \rangle - \bbE_x(A)
=0$, and so we have the identity
\[
\Delta_x(A) = \|Ax - \bbE_x(A)x\|^2.
\label{eq:var-eigenvector}
\]
In particular, this shows that $x$ is an eigenvector of $A$ if and only if
$\Delta_x(A)=0$.

\begin{proof}[Proof of Theorem \ref{thm:Lip1} \ref{thm:Lip1:pure} $\implies$ \ref{thm:Lip1:function} when $\dim H<\infty$]
Let $A,B\in\Bsa(H)$ be observables with $A\preceq B$. Then since
$\dim H<\infty$, $\sigma(B)$ is finite and $B$ has a spectral
decomposition
$$
B = \sum_{\lambda\in\sigma(B)} \lambda P_\lambda,
$$
where each $\lambda\in\sigma(B)$ is an eigenvalue of $B$ and $P_\lambda$
is the orthogonal projection onto the corresponding eigenspace.

Let $x\in H$ be a unit eigenvector of $B$. Then since $A\preceq B$ we have
$$
0\leq\Delta_x(A)\leq\Delta_x(B)=0.
$$
Hence $x$ is also an eigenvector of $A$. Moreover, this shows that for each
$\lambda\in\sigma(B)$, every $x\in\ran P_\lambda$ is an eigenvector of $A$,
so we conclude that $\ran P_\lambda$ is an eigenspace of $A$. Set
$f(\lambda)$ to be the eigenvalue of $A$ corresponding to $P_\lambda$.
Then
$$
A = \sum_{\lambda\in\sigma(B)} f(\lambda) P_\lambda,
$$
so that $A=f(B)$.

It remains to show that $f\in\Lip_1(\sigma(B))$. To see
this, we take $\lambda_1,\lambda_2\in\sigma(B)$, $\lambda_1 \neq\lambda_2$,
and unit eigenvectors $x_1, x_2$ of $B$, with eigenvalues $\lambda_1,
\lambda_2$ respectively. Set $y=(x_1+x_2)/\sqrt{2}$ so that $\|y\|=1$.
Then for any $f:\sigma(B)\to\R$, a simple computation gives that $\bbE_y(f(B))=(f(\lambda_1)+f(\lambda_2))/2$ and
$$
\Delta_y(f(B)) = \|f(B)y - \bbE_y(f(B))y\|^2
 = \frac{1}{8}(f(\lambda_1)-f(\lambda_2))^2\|x_1-x_2\|^2.
$$ 
However, since $x_1$ and $x_2$ are orthogonal, the right hand side
is equal to
$$
\frac{1}{4}(f(\lambda_1)-f(\lambda_2))^2.
$$
It follows that if $A=f(B)\preceq B$ we must have that
$$
|f(\lambda_1)-f(\lambda_2)| = 2\sqrt{\Delta_y(f(B))}
\leq 2\sqrt{\Delta_y(B)}
= |\lambda_1-\lambda_2|.
$$
Hence $f\in\Lip_1(\sigma(B))$.
\end{proof}

\subsection{Proof of \ref{thm:Lip1:pure}~$\implies$~\ref{thm:Lip1:function}
for general $H$}

The case when $H$ is infinite dimensional is more subtle because a general
observable need not have a discrete spectrum. To deal with this we will
work with a slight generalization of the notion of an eigenvector: we will
say that a family of unit vectors $(x_\eps)\subseteq H$, defined for all
sufficiently small $\eps>0$, is an \emph{approximate eigenvector} for an
observable $A$ with \emph{approximate eigenvalue} $\lambda\in\R$ if and
only if $\|Ax_\eps-\lambda x_\eps\|\to 0$ as $\eps\to0$. It is standard that
there exists an approximate eigenvector
for $A$ with approximate eigenvalue $\lambda$ if and only if $\lambda
\in\sigma(A)$, so each observable has approximate eigenvectors. Moreover,
it is not difficult to see that if $(x_\eps)$ is an approximate
eigenvector for $A$ with approximate eigenvalue $\lambda$ and
$f:\sigma(A)\to\C$ is continuous, then $(x_\eps)$ is also an
approximate eigenvector for $f(A)$ with approximate eigenvalue $f(\lambda)$.

We start with a simple characterisation of approximate eigenvectors in
terms of the expectation and variance of an observable.

\begin{lemma}\label{lem:appox1}
Let $A\in\Bsa(H)$ be an observable. Then a family of unit vectors
$(x_\eps)\subseteq H$ is an approximate eigenvector for $A$ with
approximate eigenvalue $\lambda\in\sigma(A)$ if and only if
$\Delta_{x_\eps}(A)\to 0$ and $\bbE_{x_\eps}(A)\to\lambda$ as $\eps\to0$.
\end{lemma}

\begin{proof}
First observe that for any unit vector $x\in H$, the quantity
$\|Ax - tx\|$, $t\in\R$, is minimized at $t=\bbE_x(A)$. Consequently,
using \eqref{eq:var-eigenvector}, we see that for any $\lambda\in\R$
we have
\begin{align*}
\Delta_{x_\eps}(A) + |\bbE_{x_\eps}(A)-\lambda|^2
&= \|Ax_\eps-\bbE_{x_\eps}(A)x_\eps\|^2
+ |\langle Ax_\eps,x_\eps\rangle - \lambda|^2 \\
&\leq \|Ax_\eps-\lambda x_\eps\|^2 + |\langle Ax_\eps-\lambda x_\eps,x_\eps\rangle|^2 \\
&\leq 2\|Ax_\eps-\lambda x_\eps\|^2.
\end{align*}
In addition, applying the triangle inequality together with the elementary
estimate $(a+b)^2\leq 2(a^2+b^2)$ we see that
\begin{align*}
\|Ax_\eps-\lambda x_\eps\|^2
&\leq 2\left(\|Ax_\eps-\bbE_{x_\eps}(A)x_\eps\|^2 +
\|\bbE_{x_\eps}(A)x_\eps-\lambda x_\eps\|^2\right) \\
&=2\left(\Delta_{x_\eps}(A) + |\bbE_{x_\eps}(A)-\lambda|^2\right).
\end{align*}
Combining these we have
$$
\frac{1}{2}\|Ax_\eps-\lambda x_\eps\|^2
\leq \Delta_{x_\eps}(A) + |\bbE_{x_\eps}(A)-\lambda|^2
\leq 2\|Ax_\eps-\lambda x_\eps\|^2.
$$
Thus $\|Ax_\eps-\lambda x_\eps\|\to 0$ if and only if $\Delta_{x_\eps}(A)\to 0$
and $\bbE_{x_\eps}(A)\to\lambda$ as $\eps\to0$.
\end{proof}

We will also need the following:

\begin{lemma}\label{lem:approx2}
Let $A\in\Bsa(H)$ be an observable and let $(x_\eps),(y_\eps)\subseteq H$
be approximate eigenvectors for $A$ with approximate eigenvalues
$\lambda$ and $\mu$ respectively, $\lambda\neq\mu$. Then for all
sufficiently small $\eps>0$, $x_\eps$ and $y_\eps$ are linearly
independent. Moreover, for any $\alpha,\beta\in\R\setminus\{0\}$ and
$(z_\eps)\subseteq H$ defined by
$$
z_\eps := \frac{\alpha x_\eps + \beta y_\eps}{\|\alpha x_\eps + \beta y_\eps\|},
\quad
\eps>0\;\;\text{sufficiently small},
$$
we have
\[
\Delta_{z_\eps}(A)\to\frac{\alpha^2\beta^2}{(\alpha^2+\beta^2)^2}
(\lambda-\mu)^2 \quad\text{as}\quad \eps\to0.
\label{eq:approx2-result}
\]
\end{lemma}

\begin{proof}
First we observe that since
\begin{align*}
|\jap{x_\eps,y_\eps}| &=
\frac{1}{|\lambda-\mu|}|\jap{\lambda x_\eps-Ax_\eps,y_\eps} + \jap{x_\eps,Ay_\eps-\mu y_\eps}| \\
&\leq \frac{1}{|\lambda-\mu|} (\|Ax_\eps-\lambda x_\eps\| + \|Ay_\eps - \mu y_\eps\|),
\end{align*}
we have that $\jap{x_\eps,y_\eps}\to 0$ as $\eps\to 0$, which proves
that $x_\eps$ and $y_\eps$ are linearly independent for all sufficiently
small $\eps>0$.

In proving \eqref{eq:approx2-result}, we assume without loss of
generality that $\alpha^2+\beta^2=1$. Then
$\|\alpha x_\eps + \beta y_\eps\|^2
= 1+2\alpha\beta\Re\jap{x_\eps,y_\eps}\to 1$ as $\eps\to0$.
Clearly we also have that
$$\|A(\alpha x_\eps + \beta y_\eps)\|^2
= \alpha^2\|Ax_\eps\|^2 + \beta^2\|Ay_\eps\|^2 + 2\alpha\beta\Re\jap{Ax_\eps,Ay_\eps}.$$
Observe that
\begin{align*}
|\jap{Ax_\eps,Ay_\eps}| &=
|\jap{Ax_\eps-\lambda x_\eps,Ay_\eps} + \lambda\jap{x_\eps,Ay_\eps-\mu y_\eps} + \lambda\mu\jap{x_\eps,y_\eps}| \\
&\leq \|A\|\|Ax_\eps-\lambda x_\eps\| + |\lambda|\|Ay_\eps - \mu y_\eps\|+ |\lambda\mu\jap{x_\eps,y_\eps}|,
\end{align*}
and so $|\jap{Ax_\eps,Ay_\eps}|\to 0$ as $\eps\to0$. Since we
obviously have $\|Ax_\eps\|\to|\lambda|$ and $\|Ay_\eps\|\to|\mu|$, we
conclude that
\[
\lim_{\eps\to0} \|Az_\eps\|^2
= \lim_{\eps\to0} \|A(\alpha x_\eps + \beta y_\eps)\|^2
=\alpha^2\lambda^2+\beta^2\mu^2.
\label{lem:approx2-norm}
\]

Next we evaluate $\lim_{\eps\to0}\bbE_{z_\eps}(A)$.
Since $\|\alpha x_\eps + \beta y_\eps\|\to 1$ as $\eps\to0$,
$$
\lim_{\eps\to0}\bbE_{z_\eps}(A)
=\lim_{\eps\to0}\jap{A(\alpha x_\eps + \beta y_\eps), \alpha x_\eps + \beta y_\eps}.
$$
Then since
$$
\jap{A(\alpha x_\eps + \beta y_\eps), \alpha x_\eps + \beta y_\eps}
 =\alpha^2\bbE_{x_\eps}(A) + \beta^2\bbE_{y_\eps}(A) + 2\alpha\beta\Re\jap{Ax_\eps,y_\eps},
$$
and $\jap{Ax_\eps,y_\eps}=\jap{Ax_\eps-\lambda x_\eps,y_\eps}+\lambda\jap{x_\eps,y_\eps}$
which tends to $0$, we see that
\[
\lim_{\eps\to0}\bbE_{z_\eps}(A)
=\lim_{\eps\to0}(\alpha^2\bbE_{x_\eps}(A) + \beta^2\bbE_{y_\eps}(A))
= \alpha^2\lambda+\beta^2\mu.
\label{lem:approx2-expectation}
\]

Combining \eqref{lem:approx2-norm} and \eqref{lem:approx2-expectation}.
We conclude that
\begin{align*}
\lim_{\eps\to0}\Delta_{z_\eps}(A)
&= \lim_{\eps\to0}\|Az_\eps\|^2 - \lim_{\eps\to0}\bbE_{z_\eps}(A)^2 \\
&= \alpha^2\lambda^2+\beta^2\mu^2 - (\alpha^2\lambda+\beta^2\mu)^2 \\
&= \alpha^2(1-\alpha^2)\lambda^2 + \beta^2(1-\beta^2)\mu^2
-2\alpha^2\beta^2\lambda\mu \\
&=\alpha^2\beta^2(\lambda-\mu)^2.
\end{align*}

\end{proof}

Now we can return to the proof of Theorem \ref{thm:Lip1}. As such, for the
remainder of this section we fix observables $A,B\in\Bsa(H)$ such that
$A\preceq B$.

\begin{lemma}\label{lem:lip}
There exists $f\in\Lip_1(\sigma(B))$ such that if $(x_\eps)$ is an
approximate eigenvector for $B$ with approximate eigenvalue $\lambda$,
then $(x_\eps)$ is also an approximate eigenvector for $A$ with
approximate eigenvalue $f(\lambda)$.
\end{lemma}

Throughout the proof, we will routinely use the characterization of approximate eigenvectors given in Lemma \ref{lem:appox1}.

\begin{proof}
Fix $\lambda\in\sigma(B)$ and let $(x_\eps)$ and $(y_\eps)$ be approximate
eigenvectors for $B$ with approximate eigenvalue $\lambda$. First we will show that
if $(x_\eps)$ and $(y_\eps)$ are also approximate eigenvectors for $A$, then they
must have the same approximate eigenvalue for $A$. To this end,
suppose that $(x_\eps)$ and $(y_\eps)$ are approximate eigenvectors for $A$
with approximate eigenvalues $\mu$ and $\eta$ respectively. Assume
$\mu\neq\eta$. Then by Lemma \ref{lem:approx2}, $x_\eps+y_\eps\neq0$
for sufficiently small $\eps>0$, and for these $\eps$ we set
\[
z_\eps = \frac{x_\eps + y_\eps}{\|x_\eps +  y_\eps\|}.
\label{eq:lem:lip}
\]
Clearly $(z_\eps)$ is an approximate eigenvector for $B$. Together
with Lemma \ref{lem:approx2} this implies that
$$
0<(\mu-\eta)^2
=4\lim_{\eps\to 0}\Delta_{z_\eps}(A)
\leq 4\lim_{\eps\to 0}\Delta_{z_\eps}(B) = 0,
$$
which is a contradiction. Hence $\mu=\eta$.

Next, we will show that each approximate eigenvector $(x_\eps)$ for $B$ is also
an approximate eigenvector for $A$. Observe that since $\Delta_{x_\eps}(A) \leq
\Delta_{x_\eps}(B)\to 0$ as $\eps\to 0$, we only need to show that $\bbE_{x_\eps}(A)$ converges as $\eps\to 0$. Moreover, since $\bbE_{x_\eps}(A)$ is bounded as a
function of $\eps$, it is sufficient to show that if $(\eps_n), (\eps_n')$ are sequences converging to $0$, and both $\bbE_{x_{\eps_n}}(A)$
and $\bbE_{x_{\eps'_n}}(A)$ converge as $n\to\infty$ then
$$
\lim_{n\to\infty} \bbE_{x_{\eps_n}}(A)
= \lim_{n\to\infty} \bbE_{x_{\eps'_n}}(A).
$$
Construct an approximate eigenvector $(y_\eps)$ by setting $y_\eps=x_{\eps_n}$ for
$\eps_{n+1}<\eps \leq\eps_{n}$. Construct $(y'_\eps)$ similarly using $(\eps_n')$.
Then clearly $(y_\eps)$ and $(y'_\eps)$ are approximate eigenvectors for both $A$
and $B$. Moreover, since they have the same approximate eigenvalue for $B$, they
must also have the same approximate eigenvalue for $A$. Hence
$$
\lim_{n\to\infty} \bbE_{x_{\eps_n}}(A)
= \lim_{\eps\to0} \bbE_{y_{\eps}}(A)
= \lim_{\eps\to0} \bbE_{y'_{\eps}}(A)
= \lim_{n\to\infty} \bbE_{x_{\eps'_n}}(A).
$$

For $\lambda\in\sigma(B)$, we take an approximate eigenvector $(x_\eps)$
for $B$ with approximate eigenvalue $\lambda$ and set
$$
f(\lambda):=\lim_{\eps\to0}\bbE_{x_\eps}(A).
$$
Observe that this does not depend on the choice of $(x_\eps)$ and so
$f:\sigma(B)\to\R$ is well-defined and each approximate eigenvector
for $B$ with approximate eigenvalue $\lambda$ is also an approximate
eigenvector for $A$ with approximate eigenvalue $f(\lambda)$.

Finally, we show that $f\in\Lip_1(\sigma(B))$; this argument is the same
as in the finite dimensional case. Take $\lambda,\mu\in\sigma(B)$,
$\lambda\neq\mu$, and approximate eigenvectors $(x_\eps)$ and $(y_\eps)$
for $B$ with approximate eigenvalues $\lambda$ and $\mu$, respectively.
Let $z_\eps$ be given by \eqref{eq:lem:lip}. Then it
follows from Lemma \ref{lem:approx2} that
$$
(f(\lambda)-f(\mu))^2
=4\lim_{\eps\to0}\Delta_{z_\eps}(A)
\leq 4\lim_{\eps\to0}\Delta_{z_\eps}(B)
=(\lambda-\mu)^2.
$$
We conclude that $|f(\lambda)-f(\mu)|\leq|\lambda-\mu|$.
\end{proof}

To complete the proof of Theorem \ref{thm:Lip1} we will need to
consider the decomposition of $H$ as an orthogonal sum of cyclic subspaces
for $B$. To be precise, we take a family $\{H_\kappa\}$ of closed $B$-invariant
subspaces of $H$ such that $H=\bigoplus_\kappa H_\kappa$ and
$B|_{H_\kappa}$ is cyclic. The existence of the family $\{H_{\kappa}\}$ is
well-known. Let $\psi_\kappa\in H_\kappa$
be a unit vector such that
$H_{\kappa} = \overline{\mathrm{span}}\{B^j\psi_\kappa:j\geq0\}$ and set
$\mu_\kappa$ to be the measure
$$
\mu_\kappa: \Borel(\R)\to\R, \;\;
\mu_\kappa(\omega)=\jap{E_B(\omega)\psi_\kappa,\psi_\kappa},
$$
where as before $E_B$ is the spectral measure of $B$ ($\mu_\kappa$
is just the measure $\mu_{B, \psi_\kappa}$ defined in \eqref{eq:Born}).
Let $U_\kappa:H_\kappa \to L^2(\mu_\kappa)$ be the unitary operator
satisfying
$$
U_\kappa(p(B)\psi_\kappa)=p \quad
\text{for each polynomial $p$}.
$$
Observe that for each $\omega\in\Borel(\R)$,
$U_\kappa(E_B(\omega)\psi_\kappa)=\chi_\omega$, where $\chi_\omega$
is the characteristic function of $\omega$.

\begin{lemma}\label{lem:LDT-vectors}
Let $x=\sum_\kappa x_\kappa \in H$, where $x_\kappa\in H_\kappa$. Then for each $\kappa$,
and for $\mu_\kappa$-almost every $\lambda\in\sigma(B)$ we have
$$
\lim_{\eps\to 0} \frac{1}{\mu_\kappa((\lambda-\eps, \lambda+\eps))}
\jap{x, E_B((\lambda-\eps, \lambda+\eps))\psi_\kappa}
= U_\kappa(x_\kappa)(\lambda).
$$
Consequently, $x=0$ if and only if for every $\kappa$,
$$
\lim_{\eps\to 0} \frac{1}{\mu_\kappa((\lambda-\eps, \lambda+\eps))}
\jap{x, E_B((\lambda-\eps, \lambda+\eps))\psi_\kappa} = 0
$$
$\mu_\kappa$-almost everywhere.
\end{lemma}

\begin{proof}
Since $E_B((\lambda-\eps, \lambda+\eps))\psi_\kappa\in H_\kappa$,
$$
\jap{x, E_B((\lambda-\eps, \lambda+\eps))\psi_\kappa}
=\jap{x_\kappa, E_B((\lambda-\eps, \lambda+\eps))\psi_\kappa}.
$$
As $U_\kappa$ is unitary, the following holds for every point $\lambda$ in the closed support of $\mu_\kappa$:
\begin{multline*}
\frac{1}{\mu_\kappa((\lambda-\eps, \lambda+\eps))}
\jap{x_\kappa, E_B((\lambda-\eps, \lambda+\eps))\psi_\kappa} \\
=
\frac{1}{\mu_\kappa((\lambda-\eps, \lambda+\eps))}
\int_{(\lambda-\eps, \lambda+\eps)} U_\kappa(x_\kappa) \dm\mu_\kappa.
\end{multline*}
By a generalised version of the Lebesgue Differentiation Theorem (for positive regular Borel measures)
\cite[Theorem 5.8.8]{Bog1}, the right hand side converges to 
$U_\kappa(x_\kappa)(\lambda)$ for $\mu_\kappa$-almost every $\lambda$.
\end{proof}

In light of Lemma \ref{lem:LDT-vectors}, for each $\kappa$, $\lambda\in\sigma(B)$
and $\eps>0$ let us introduce the vector
$$
\varphi^\kappa_{\lambda, \eps}:=
\frac{1}{\sqrt{\mu_\kappa((\lambda-\eps, \lambda+\eps))}}
E_B((\lambda-\eps, \lambda+\eps))\psi_\kappa \in H_\kappa.
$$
Note that $\|\varphi^\kappa_{\lambda, \eps}\|=1$ and moreover, since
$$
\|(B- \lambda I)\varphi^\kappa_{\lambda, \eps}\|^2
= \frac{1}{\mu_\kappa((\lambda-\eps, \lambda+\eps))}
\int_{(\lambda-\eps, \lambda+\eps)} |t-\lambda|^2 \dm\mu_\kappa(t),
$$
we have the estimate
\[
\|(B- \lambda I)\varphi^\kappa_{\lambda, \eps}\|
= O(\eps) \;\; \text{as} \;\; \eps\to 0.
\label{eq:eps-bigo}
\]
In particular, $(\varphi^\kappa_{\lambda, \eps})$ is an approximate eigenvector
for $B$, as $\eps\to 0$, with approximate eigenvalue $\lambda$.

We will also need the following measure-theoretic fact:
for $\mu_\kappa$-almost all $\lambda$
$$
\lim_{\eps\to 0} \frac{\eps}{\mu_\kappa((\lambda-\eps, \lambda+\eps))}
$$
is well-defined and finite; see e.g. \cite[Theorem 5.8.8]{Bog1}.
And so for $\mu_\kappa$-almost every $\lambda$,
\[
\frac{1}{\mu_\kappa((\lambda-\eps, \lambda+\eps))}
= O(1/\eps), \;\; \text{as} \;\; \eps\to 0.
\label{eq:mu-bigo}
\]

We are now in a position to complete the proof of Theorem \ref{thm:Lip1}.

\begin{proof}[Proof of Theorem \ref{thm:Lip1} \ref{thm:Lip1:pure} $\implies$ \ref{thm:Lip1:function}.]

Let $f\in\Lip_1(\sigma(B))$ be the function given by Lemma \ref{lem:lip}.
We will show that for each $\kappa$ and each unit vector $x\in H$,
$$
\lim_{\eps\to 0}
\frac{\jap{Ax,\varphi^\kappa_{\lambda, \eps}}}{\sqrt{\mu_\kappa((\lambda-\eps, \lambda+\eps))}}
=
\lim_{\eps\to 0}
\frac{\jap{f(B)x,\varphi^\kappa_{\lambda, \eps}}}{\sqrt{\mu_\kappa((\lambda-\eps, \lambda+\eps))}}
$$
for $\mu_\kappa$-almost every $\lambda\in\sigma(B)$. It will then follow from
Lemma \ref{lem:LDT-vectors} that $A=f(B)$.

First observe that since $(\varphi^\kappa_{\lambda, \eps})$ is an approximate
eigenvector for $B$ with approximate eigenvalue $\lambda$, it is also an
approximate eigenvector for both $A$ and $f(B)$ with approximate eigenvalue $f(\lambda)$; for $f(B)$ this is clear and for $A$ this follows from Lemma
\ref{lem:lip}.

Take a unit vector $x\in H$ and write $x=\sum_\kappa x_\kappa$, where
$x_\kappa\in H_\kappa$. Then for any $\eps>0$ we have
\begin{align*}
\frac{\jap{Ax,\varphi^\kappa_{\lambda, \eps}}}{\sqrt{\mu_\kappa((\lambda-\eps, \lambda+\eps))}}
&=
\frac{\jap{x,(A-\jap{A\varphi^\kappa_{\lambda, \eps},
\varphi^\kappa_{\lambda, \eps}}I)\varphi^\kappa_{\lambda, \eps}} }
{\sqrt{\mu_\kappa((\lambda-\eps, \lambda+\eps))}} \\
&+ \frac{\jap{A\varphi^\kappa_{\lambda, \eps},\varphi^\kappa_{\lambda, \eps}}
\jap{x,\varphi^\kappa_{\lambda, \eps}}}
{\sqrt{\mu_\kappa((\lambda-\eps, \lambda+\eps))}}.
\end{align*}
Observe that by \eqref{eq:var-eigenvector} for $\mu_\kappa$-almost every $\lambda$,
\begin{align*}
\left|\frac{\jap{x,(A-\jap{A\varphi^\kappa_{\lambda, \eps},\varphi^\kappa_{\lambda, \eps}}I)
\varphi^\kappa_{\lambda, \eps}}}{\sqrt{\mu_\kappa((\lambda-\eps, \lambda+\eps))}}
\right|
&\leq
\sqrt{\frac{\Delta_{\varphi^\kappa_{\lambda, \eps}}(A)}
{\mu_\kappa((\lambda-\eps, \lambda+\eps))}} \\
&\leq
\sqrt{\frac{\Delta_{\varphi^\kappa_{\lambda, \eps}}(B)}
{\mu_\kappa((\lambda-\eps, \lambda+\eps))}} \\
&\leq
\frac{\|(B- \lambda I)\varphi^\kappa_{\lambda, \eps}\|}
{\sqrt{\mu_\kappa((\lambda-\eps, \lambda+\eps))}} \\
&= O(\sqrt{\eps}) \;\;\text{as} \;\; \eps\to 0,
\end{align*}
where the last estimate followed from combining \eqref{eq:eps-bigo} and
\eqref{eq:mu-bigo}. Therefore,
\begin{align*}
\lim_{\eps\to 0}
\frac{\jap{Ax,\varphi^\kappa_{\lambda, \eps}}}{\sqrt{\mu_\kappa((\lambda-\eps, \lambda+\eps))}}
&=
\lim_{\eps\to 0}
\frac{\jap{A\varphi^\kappa_{\lambda, \eps},\varphi^\kappa_{\lambda, \eps}}
\jap{x,\varphi^\kappa_{\lambda, \eps}}}
{\sqrt{\mu_\kappa((\lambda-\eps, \lambda+\eps))}} \\
&=
\lim_{\eps\to 0}
\jap{A\varphi^\kappa_{\lambda, \eps},\varphi^\kappa_{\lambda, \eps}}
\lim_{\eps\to 0}
\frac{\jap{x,\varphi^\kappa_{\lambda, \eps}}}
{\sqrt{\mu_\kappa((\lambda-\eps, \lambda+\eps))}} \\
&= f(\lambda) \cdot U_\kappa x_\kappa(\lambda)
\quad \text{for $\mu_\kappa$-almost all $\lambda$}.
\end{align*}
The same argument with $f(B)$ instead of $A$ also shows that
$$
\lim_{\eps\to 0}
\frac{\jap{f(B)x,\varphi^\kappa_{\lambda, \eps}}}{\sqrt{\mu_\kappa((\lambda-\eps, \lambda+\eps))}}
= f(\lambda)\cdot U_\kappa x_\kappa(\lambda)
\quad \text{for $\mu_\kappa$-almost all $\lambda$},
$$
which completes the proof.
\end{proof}


\section{Proof of Theorem \ref{thm:pres}}\label{sec:pres}

We start with proving a few lemmas.

\begin{lemma}\label{lem:comm}
Assume that $A,B\in\Bsa(H)$ are two observables and that the spectrum of $A$ contains finitely many elements: $\#\sigma(A)<\infty$. In that case the following two conditions are equivalent:
\begin{enumerate}[label=(\roman*)]
    \item\label{lem:comm:1} $A$ and $B$ commute;
    \item\label{lem:comm:2} there exists an observable $C\in\Bsa(H)$ such that $A\preceq C$ and $B\preceq C$.
\end{enumerate}
\end{lemma}

\begin{proof}
\ref{lem:comm:2}$\implies$\ref{lem:comm:1} is a trivial consequence
of Theorem \ref{thm:Lip1}. As for the reverse, assume \ref{lem:comm:1} and that $\sum_{j=1}^n\lambda_jP_j$ is the spectral decomposition of $A$ with some projections $P_j$ and real numbers $\lambda_j$. Since $A$ and $B$ commute, $B$ commutes with each $P_j$, hence $B=\oplus_{j=1}^nB_j$ with some $B_j\in\Bsa(\ran P_j)$. Choose positive numbers $\beta,\tau$ such that $\sigma(B_j)\subseteq [-\tau,\tau]$ for all $j$, and $\beta > 4\tau+\diam(\sigma(A))$. Define the observable
$$
C := \bigoplus_{j=1}^n (B_j + j\beta),
$$
whose spectrum is the disjoint union $\cup_{j=1}^n(\sigma(B_j)+j\beta)$. Consider the functions
$$
f\colon \bigcup_{j=1}^n[-\tau+j\beta,\tau+j\beta]
\to\R, \quad f(x) = \lambda_j \;\;\text{if } |x-j\beta|\leq\tau,
$$
$$
g\colon \bigcup_{j=1}^n[-\tau+j\beta,\tau+j\beta]
\to\R, \quad g(x) = x-j\beta \;\;\text{if } |x-j\beta|\leq\tau.
$$
Clearly, $f$ and $g$ are Lipschitz functions with constant $1$, $A=f(C)$ and $B=g(C)$. This completes the proof.
\end{proof}

Recall that in case of the Loewner order any observable $A\in\Bsa(H)$ is completely determined by the set of those obsevables whose spectrum contains at most $2$ elements and that are below $A$. The following lemma shows that a similar statement does not hold for the variance order.

\begin{lemma}\label{lem:3-spec}
Assume that $\varclass{A} = \varclass{0P_0 + t_1P_1 + (t_1+t_2)P_2}$ where $t_2,t_1>0$ and $P_0,P_1,P_2$ are non-trivial pairwise orthogonal projections such that $I=P_0+P_1+P_2$. 
Then
\begin{align}\label{eq:3-spec-2-below}
    &\{B\in\Bsa(H)\colon B\preceq A, \#\sigma(B)\leq 2\} \nonumber\\
    &\qquad = \left(
    \bigcup_{0\leq q\leq t_1}\varclass{qP_0}\right)\bigcup
    \left(\bigcup_{0\leq r\leq\min(t_1,t_2)}\varclass{rP_1}\right)\bigcup
    \left(\bigcup_{0\leq s\leq t_2}\varclass{sP_2}\right).
\end{align}
\end{lemma}

\begin{proof}
One only needs to determine the set of those functions $f\in\Lip_1(\{0,t_1,t_1+t_2\})$ such that $\#\{f(0),f(t_1),f(t_1+t_2)\}\leq2$. If $f$ is not a constant function, then two of these values coincide. For instance, if $f(0)=f(t_1+t_2)$, then we must have $|f(0)-f(t_1)|\leq \min(t_1,t_2)$, which gives the middle union on the right-hand side of \eqref{eq:3-spec-2-below}.
\end{proof}

Using the notation of the above lemma, assume that $t_2 > t_1>0$, $\widehat{A}\in\Bsa(H)$, $\#\sigma(\widehat{A})=3$ and
\begin{equation}\label{eq:corr-cond}
    \{B\in\Bsa(H)\colon B\preceq A, \#\sigma(B)\leq 2\} = \{B\in\Bsa(H)\colon B\preceq \widehat{A}, \#\sigma(B)\leq 2\}.
\end{equation}
By Lemma \ref{lem:3-spec}, this holds if and only if either 
\begin{equation}\label{eq:corr-c1}
    \varclass{\widehat{A}}=\varclass{A}=\varclass{0P_0 + t_1P_1 + (t_1+t_2)P_2},    
\end{equation}
or
\begin{equation}\label{eq:corr-c2}
    \varclass{\widehat{A}}=\varclass{0P_2 + t_2P_0 + (t_1+t_2)P_1}.
\end{equation}
On the other hand, if $t_2 = t_1>0$, then \eqref{eq:corr-cond} is equivalent to either \eqref{eq:corr-c1}, or \eqref{eq:corr-c2}, or
\begin{equation}\label{eq:corr-c3}
    \varclass{\widehat{A}}=\varclass{0P_1 + t_1P_2 + 2t_1P_0}.
\end{equation}

Nonetheless, the variance order satisfies at least the following property, which will be useful in the proof of Theorem \ref{thm:pres}.

\begin{lemma}\label{lem:for-ind}
Assume that $n\in\N, n\geq 4$, $A_1,A_2\in\Bsa(H)$ are two observables such that $\#\sigma(A_1)=\#\sigma(A_2)=n$, and
\begin{equation}\label{eq:info}
    \{B\in\Bsa(H)\colon B\preceq A_1, \#\sigma(B)<n\} = \{B\in\Bsa(H)\colon B\preceq A_2, \#\sigma(B)<n\}.
\end{equation}
Then we have $\varclass{A_1}=\varclass{A_2}$.
\end{lemma}

\begin{proof}
Consider an $A\in\Bsa(H)$ with an $n$-element spectrum $\{\lambda_1,\dots,\lambda_n\}$. Let $\omega\subseteq\sigma(A),\; \omega\neq\emptyset,\; \omega\neq\sigma(A)$, and $f$ a function that is constant on both $\omega$ and $\sigma(A)\setminus\omega$. Clearly, $f\in\Lip_1(\sigma(A))$ if and only if the distance between its two different values is at most $t_\omega:=\min\{|\lambda-\mu|\colon \lambda\in\omega, \mu\in\sigma(A)\setminus\omega\}$. Therefore, 
$$
\{B\in\Bsa(H)\colon B\preceq A, \#\sigma(B)\leq 2\} = \bigcup\{\varclass{t E_A(\omega)}\colon \omega\subseteq\sigma(A), 0\leq t\leq t_\omega\}.
$$
From here we conclude that the information \eqref{eq:info} implies $A_i = \sum_{j=1}^n \lambda_j^iP_j$ with some pairwise orthogonal projections $\{P_j\}_{j=1}^n$ such that $P_1+\dots+P_n = I$ and some distinct real numbers $\{\lambda_j^i\}_{j=1}^n$ $(i=1,2)$.

Next, for each $j$ choose an arbitrary unit vector $x_j\in\ran P_j$. For any $j,k\in\{1,2,\dots,n\}$, $j<k$, define the following quantity that clearly does not depend on $i$:
\begin{align*}
    q_{j,k} &:= \sup\{|\langle Bx_j,x_j\rangle - \langle Bx_k,x_k\rangle|\colon B\preceq A_i, \#\sigma(B)<n\} \\
    &= \sup\{|\langle f(A_i)x_j,x_j\rangle - \langle f(A_i)x_k,x_k\rangle|\colon f\in\Lip_1(\sigma(A_i)), \# f(\sigma(A_i))<n\} \\
    &= \sup\{|f(\lambda_j^i)-f(\lambda_k^i)|\colon f\in\Lip_1(\sigma(A_i)), \# f(\sigma(A_i))<n\}.
\end{align*}
Note that $q_{j,k} \leq |\lambda_j^i-\lambda_k^i|$ $(i=1,2)$. First, if $|\lambda_j^i-\lambda_k^i| < \diam(\sigma(A_i))$, then 
$$
    \Lip_1(\sigma(A_i)) \ni f(x) = \left\{ \begin{matrix}
        \min(\lambda_j^i,\lambda_k^i), & \text{if } x<\min(\lambda_j^i,\lambda_k^i) \\
        x, & \text{if } \min(\lambda_j^i,\lambda_k^i)\leq x\leq \max(\lambda_j^i,\lambda_k^i) \\
        \max(\lambda_j^i,\lambda_k^i), & \text{if } x>\max(\lambda_j^i,\lambda_k^i)
    \end{matrix} \right. 
$$
and $\# f(\sigma(A_i))<n$. Therefore $q_{j,k} = |\lambda_j^i-\lambda_k^i|$ follows in that case.
As for the $|\lambda_j^i-\lambda_k^i| = \diam(\sigma(A_i))$ case, note that for any $f\in\Lip_1(\sigma(A_i))$ with $f(\lambda_l^i)=f(\lambda_m^i)$ we have $|f(\lambda_j^i)-f(\lambda_k^i)| \leq \diam(\sigma(A_i)) - |\lambda_l^i-\lambda_m^i|$. Hence we obtain 
\begin{equation}\label{eq:max}
q_{j,k} = \diam(\sigma(A_i)) - \min\{|\lambda-\mu|\colon\lambda,\mu\in\sigma(A_i),\lambda\neq\mu\}.
\end{equation}

We now show that the metric structure of $\sigma(A_i)$ is encoded in the information $\{q_{j,k}\colon 1\leq j<k\leq n\}$.
Observe that the right-hand side of \eqref{eq:max} is the maximum of $\{q_{j,k}\colon 1\leq j<k\leq n\}$, and that at most three of these numbers are equal to this maximum. 
More precisely, assume $\lambda_{j_1}^1 < \lambda_{j_2}^1 < \dots < \lambda_{j_{n-1}}^1 < \lambda_{j_{n}}^1$ are the spectrum points of $A_1$ (similar conclusions hold for $A_2$, though the indexing might be different). We have three cases: 
\begin{itemize}
	\item First, $|\lambda_{j_1}^1-\lambda_{j_2}^1| = |\lambda_{j_{n-1}}^1-\lambda_{j_{n}}^1| = \min\{|\lambda-\mu|\colon\lambda,\mu\in\sigma(A_1),\lambda\neq\mu\}$ holds if and only if $\max\{q_{j,k}\colon 1\leq j<k\leq n\}$ is attained exactly for three pairs, in which case it is attained for $(j,k) = (j_1,j_{n-1}), (j_1,j_{n}), (j_2,j_{n})$. 
	\item Second, either $|\lambda_{j_1}^1-\lambda_{j_2}^1| > |\lambda_{j_{n-1}}^1-\lambda_{j_{n}}^1| = \min\{|\lambda-\mu|\colon\lambda,\mu\in\sigma(A_1),\lambda\neq\mu\}$ or $|\lambda_{j_{n-1}}^1-\lambda_{j_{n}}^1| > |\lambda_{j_1}^1-\lambda_{j_2}^1| = \min\{|\lambda-\mu|\colon\lambda,\mu\in\sigma(A_1),\lambda\neq\mu\}$ holds if and only if $\max\{q_{j,k}\colon 1\leq j<k\leq n\}$ is attained exactly for two pairs, in which case it is attained for $(j,k) = (j_1,j_{n-1}), (j_1,j_{n})$, or $(j,k) = (j_2,j_{n}), (j_1,j_{n})$, respectively. 
	\item Third, $|\lambda_{j_{n-1}}^1-\lambda_{j_{n}}^1|, |\lambda_{j_1}^1-\lambda_{j_2}^1| > \min\{|\lambda-\mu|\colon\lambda,\mu\in\sigma(A_1),\lambda\neq\mu\}$ holds if and only if $\max\{q_{j,k}\colon 1\leq j<k\leq n\}$ is attained uniquely for $(j,k) = (j_1,j_{n})$.
\end{itemize}
In light of this, we conclude the following:
\begin{itemize}
	\item If $\max\{q_{j,k}\colon 1\leq j<k\leq n\}$ is attained exactly for three pairs, say for $(j,k) = (k_1,k_{n-1}), (k_1,k_{n}), (k_2,k_{n})$, then for each $i$ we must have $\{\min\sigma(A_i),\max\sigma(A_i)\}=\{\lambda_{k_1}^i,\lambda_{k_n}^i\}$, as these are the only indexes appearing twice. Moreover, $|\lambda_{k_1}^i-\lambda_{k_n}^i| = |\lambda_{k_1}^i-\lambda_{k_{n-1}}^i| + |\lambda_{k_{n-1}}^i-\lambda_{k_n}^i| = q_{k_1,k_{n-1}}+q_{k_{n-1},k_n}$ $(i=1,2)$.
	\item If $\max\{q_{j,k}\colon 1\leq j<k\leq n\}$ is attained exactly for two pairs, say for $(j,k) = (k_1,k_{n-1}), (k_1,k_{n})$, then for each $i$ we must have $\lambda_{k_1}^i\in\{\min\sigma(A_i),\max\sigma(A_i)\}$, as this index is the only one appearing twice. Moreover, $|\lambda_{k_1}^i-\lambda_{k_m}^i| = |\lambda_{k_1}^i-\lambda_{j}^i|+|\lambda_{j}^i-\lambda_{k_m}^i| = q_{k_1,j}+q_{j,k_m}$ for any $j\in\{1,2,\dots,n\}\setminus\{k_1,k_{n-1},k_{n}\}$ ($m\in\{n-1,n\}$, $i\in\{1,2\}$).
	\item If $\max\{q_{j,k}\colon 1\leq j<k\leq n\}$ is attained uniquely for $(j,k) = (k_1,k_{n})$, then for each $i$ we must have $\{\min\sigma(A_i),\max\sigma(A_i)\}=\{\lambda_{k_1}^i,\lambda_{k_n}^i\}$. Moreover, $|\lambda_{k_1}^i-\lambda_{k_n}^i| = |\lambda_{k_1}^i-\lambda_{j}^i|+|\lambda_{j}^i-\lambda_{k_n}^i| = q_{k_1,j}+q_{j,k_n}$ for any $j\in\{1,2,\dots,n\}\setminus\{k_1,k_{n}\}$ $(i=1,2)$.
\end{itemize}
In all of the above scenarios we obtain $|\lambda_j^1-\lambda_k^1| = |\lambda_j^2-\lambda_k^2|$ for all $1\leq j < k\leq n$. Hence $A_2 = \pm A_1+cI$ with some real $c$.
\end{proof}

We point out that the above lemma does not hold for general operators. Namely, for any observable $A\in\Bsa(H)$ with $\sigma(A) = [0,1]$ the only observables $B\in\Bsa(H)$ that satisfies $\#\sigma(B)<\infty$ and $B\preceq A$ are the scalar multiples of the identity, since the only Lipschitz functions on $[0,1]$ that have a finite range are the constant functions.

For any projection $P$ we introduce the notation $P^\perp := I-P$.
Recall the following theorem which we shall use in the proof of Theorem \ref{thm:pres}. The proof can be found in \cite[Theorem 4.3]{GeSeCMP} or \cite[Theorem 2.8]{MoS2}.

\begin{theorem}\label{thm:ProjComm}
Let $H$ be a Hilbert space of dimension at least $3$ and $\phi\colon \Proj(H)\to \Proj(H)$ be a bijective mapping that preserves commutativity in both directions, i.e.
\begin{equation}\label{eq:ProjComm}
	PQ = QP \;\; \iff \;\; \phi(P)\phi(Q) = \phi(Q)\phi(P) \qquad (P,Q \in \Proj(H)).
\end{equation}
Then there exists a unitary or antiunitary operator $U\colon H\to H$ such that
\begin{equation*}
	\phi(P) \in \{ UPU^*, UP^\perp U^* \} \qquad (P \in \Proj(H)).
\end{equation*}
\end{theorem}

Now, we are in the position to prove our second main result.

\begin{proof}[Proof of Theorem \ref{thm:pres}]
First, note that the bijectivity of $\Phi$ and \eqref{eq:iff-preceq} imply that
$$
    \Phi(\varclass{A}) = \varclass{\Phi(A)} \qquad (A\in\Bsa(H)).
$$
Also, clearly for any $A\in\Bsa(H)$ we have $\Bsa(H)=\{B\colon A\preceq B\}$ if and only if $A\in\{\lambda I\colon\lambda\in\R\}=\varclass{0}$, hence we obtain
$$
    \Phi(\varclass{0}) = \varclass{0}.
$$
For the sake of clarity, from here we split our proof into several parts.
We emphasise that continuity of $\Phi$ in the infinite dimensional case is only exploited in the very last step.

\smallskip

STEP 1: \emph{We show that}
\begin{equation}\label{eq:2-element-form00}
    \Phi(\varclass{tP}) = \varclass{\xi_P(t)F(P)} \qquad (t>0, P\in\Proj(H)\setminus\{0,I\})
\end{equation}
\emph{holds with a bijective map $F\colon\Proj(H)\setminus\{0,I\}\to\Proj(H)\setminus\{0,I\}$ and a family of bijective strictly increasing functions $\{\xi_P\colon(0,\infty)\to(0,\infty)\}_{P\in\Proj(H)\setminus\{0,I\}}$ such that $F(P^\perp) = F(P)^\perp$ and $\xi_P=\xi_{P^\perp}$ for all $P$.}

Let us introduce the notation $A^\succeq := \{B\in\Bsa(H)\colon A\succeq B\}$. Obviously, we have $\Phi(A^\succeq) = \Phi(A)^\succeq$ for all $A\in\Bsa(H)$. We claim that the spectrum of $A$ contains exactly $2$ elements if and only if
\begin{equation}\label{eq:2-spec}
    A^\succeq \neq \varclass{A} \text{ and for all } B_1,B_2\in A^\succeq \text{ we have } B_1\preceq B_2 \text{ or } B_1\succeq B_2.
\end{equation}
Note that \eqref{eq:2-spec} clearly does not hold for scalar multiples of the identity.
In case $\#\sigma(A) = 2$, then $\varclass{A} = \varclass{tP}$ holds with some $t>0, P\in\Proj(H)\setminus\{0,I\}$. By Theorem \ref{thm:Lip1} it is then straightforward that 
\begin{equation}\label{eq:2-spec-succeq}
    (tP)^\succeq = (tP^\perp)^\succeq = A^\succeq = \cup_{0\leq s\leq t}\varclass{sP},    
\end{equation}
from which \eqref{eq:2-spec} follows.
For the other direction, assume that $\sigma(A)$ contains at least three different elements, say $\lambda_1<\lambda_2<\lambda_3$. Define the functions
$$
    f(x) = \left\{
    \begin{matrix}
    0 & \text{if } x\leq\lambda_2 \\
    x-\lambda_2 & \text{if } x>\lambda_2
    \end{matrix}
    \right.,
    \quad 
    g(x) = \left\{
    \begin{matrix}
    x-\lambda_2 & \text{if } x\leq\lambda_2 \\
    0 & \text{if } x>\lambda_2
    \end{matrix}
    \right.
$$
and note that $f,g\in\Lip_1(\R)$, hence $f(A),g(A)\in A^\succeq$. Notice that $f(A)\preceq g(A)$ would imply $f(x)=h(g(x))$ for all $x\in\sigma(A)$ with some $h\in\Lip_1(\R)$. However, substituting $x=\lambda_2$ and $x=\lambda_3$ shows that this cannot be satisfied. We see similarly that $f(A)\succeq g(A)$ cannot hold, and therefore that \eqref{eq:2-spec} is not fulfilled which proves our claim.
Note that this implies that if $\#\sigma(A)=2$ then 
$\#\sigma(\Phi(A))=2$, and so \eqref{eq:2-element-form00} must hold. It's clear that $F(P^\perp) = F(P)^\perp$ and $\xi_P=\xi_{P^\perp}$ for all $P$, and from \eqref{eq:2-spec-succeq} we see that each
$\xi_P$ must be a strictly increasing bijection.

\smallskip

STEP 2: \emph{We prove that there exist a unitary or antiunitary operator $U\colon H\to H$ and a family of bounded Borel functions $\{f_A\colon\sigma(A)\to\R\}_{A\in\Bsa(H)}$ such that}
\begin{equation}\label{eq:comm-form}
    \Phi(A) = Uf_A(A)U^* \qquad (A\in\Bsa(H)).
\end{equation}

Note that for all $P,Q\in\Proj(H)\setminus\{0,I\}$ we have
\begin{align*}
    PQ=QP & \iff \exists \;C\in\Bsa(H)\colon P,Q\in C^\succeq \\
    & \iff \exists \;D\in\Bsa(H)\colon \xi_P(1)F(P),\xi_Q(1)F(Q)\in D^\succeq \\
    & \iff \xi_P(1)F(P)\cdot\xi_Q(1)F(Q)=\xi_Q(1)F(Q)\cdot\xi_P(1)F(P) \\
    & \iff F(P)F(Q)=F(Q)F(P).
\end{align*}
where we used Lemma \ref{lem:comm}. Let us extend $F$ to $\Proj(H)$ by setting $F(0)=0$ and $F(I)=I$. Clearly, this $F$ also preserves commutativity, hence by Theorem \ref{thm:ProjComm} we obtain that 
\begin{equation}\label{eq:2-element-form0}
    \Phi(\varclass{tP}) = \varclass{\xi_P(t)UPU^*} \qquad (t>0, P\in\Proj(H)\setminus\{0,I\})
\end{equation}
holds with a unitary or antiunitary operator $U$.

Next, by Lemma \ref{lem:comm} we obtain that for all $A,B\in\Bsa(H)$
\begin{align*}
    AB=BA & \iff \forall \;\text{spectral projections } P \text{ of } A\colon PB=BP \\
    & \iff \forall \;\text{sp.~proj.~} P \text{ of } A, \;\exists \; C_P\in\Bsa(H)\colon P,B\in (C_P)^\succeq \\
    & \iff \forall \;\text{sp.~proj.~} P \text{ of } A, \;\exists \; D_P\colon \xi_P(1)UPU^*,\Phi(B)\in (D_P)^\succeq \\
    & \iff \forall \;\text{sp.~proj.~} P \text{ of } A\colon UPU^*\cdot \Phi(B)=\Phi(B)\cdot UPU^* \\
    & \iff A\cdot U^*\Phi(B)U = U^*\Phi(B)U\cdot A.
\end{align*}
This implies that $U^*\Phi(B)U$ is in the double commutant of $B$, therefore (as $H$ is separable) there exists a bounded Borel function $f_B\colon\sigma(B)\to\R$ such that $\Phi(B) = Uf_B(B)U^*$, which proves \eqref{eq:comm-form}. Note that we could have arrived at the same conclusion by applying \cite[Corollary 2]{MoSe-comm}, the reason we did not do that is that we found the above approach shorter and more direct.

As we have not used continuity yet, a similar statement holds for $\Phi^{-1}$, therefore we obtain 
\begin{equation}\label{eq:spec-card-pres}
    \#\sigma(A) = \#\sigma(\Phi(A)) \qquad (A\in\Bsa(H)).
\end{equation}
Observe that without loss of generality we may, and from now on will, assume that $U=I$, since $\Phi$ satisfies the required properties if and only if $A\mapsto U^*\Phi(A)U$ does. We shall show that this assumption implies the existence of a positive $\alpha$ such that $\Phi\left(\varclass{A}\right) = \varclass{\alpha A}$ for all $A$.

\smallskip

STEP 3: \emph{We show that}
\begin{equation*}
    \Phi(\varclass{tP}) = \varclass{\xi(t)P} \qquad (t>0, P\in\Proj(H)\setminus\{0,I\}),
\end{equation*}
\emph{holds with a bijective strictly increasing function $\xi\colon(0,\infty)\to(0,\infty)$.}

In order to see this, we only need to prove that $\xi_P=\xi_Q$ for all non-trivial projections. We already know this when $Q=P^\perp$. Next, assume that $P$ and $Q$ are orthogonal to each other but $P+Q\neq I$. For all $t_2\geq t_1>0$ set $R:=I-P-Q$ and $A_{t_1,t_2} := 0P+t_1Q+(t_1+t_2)R$. By \eqref{eq:3-spec-2-below} we obtain
\begin{align}\label{eq:At1t2}
    & \{C\in\Bsa(H)\colon C\preceq \Phi(A_{t_1,t_2}), \#\sigma(C)\leq 2\} \nonumber\\
    & = \Phi\left(\{B\in\Bsa(H)\colon B\preceq A_{t_1,t_2}, \#\sigma(B)\leq 2\}\right) \nonumber\\
    & = \left(\bigcup_{0\leq r\leq t_1}\left(\varclass{\xi_P(r)P}\cup\varclass{\xi_Q(r)Q}\right)\right)\bigcup\left(\bigcup_{0\leq s\leq t_2}\varclass{\xi_R(s)R}\right).
\end{align}
Moreover, since $\#\sigma(\Phi(A_{t_1,t_2}))=3$, the structure of the right-hand side of \eqref{eq:3-spec-2-below} implies that at least two of the following three quantities must coincide: $\xi_P(t_1),\xi_Q(t_1),\xi_R(t_2)$. Note that this holds for all $t_2\geq t_1>0$. Letting $t_2$ converge to $\infty$, we conclude that $\xi_P=\xi_Q$.

Now, note that if $P,Q$ are two different rank-$1$ projections, then there exists a third rank-$1$ projection $R$ which is orthogonal to both $P$ and $Q$. By the above paragraph, we therefore have $\xi_P=\xi_R=\xi_Q$. Finally, for any
$P,Q\in\Proj(H)\setminus\{0,I\}$, we take rank-$1$ projections
$P',Q'$ which are orthogonal to $P,Q$ respectively and observe
that $\xi_P=\xi_{P'}=\xi_{Q'}=\xi_Q$, which completes the proof
of this step.




\smallskip

STEP 4: \emph{We show that $\xi(t) = \alpha t$ with some positive constant $\alpha$.}

Fix three pairwise orthogonal projections $P_0,P_1,P_2$ such that $P_0+P_1+P_2=I$. We introduce the notation $A^{(i,j,k)}_{t_1,t_2} := 0P_i+t_1P_j+(t_1+t_2)P_k$ for $t_1,t_2>0$ and any permutation $(i,j,k)$ of $\{0,1,2\}$. In case when we further have $t_1=t_2$, then we shall simply write $A^{(i,j,k)}_{t_1}$ instead of $A^{(i,j,k)}_{t_1,t_1}$. Note that $\varclass{A^{(i,j,k)}_{t_1,t_2}} = \varclass{A^{(k,j,i)}_{t_2,t_1}}$.

Let $t>0$ be an arbitrary positive number.
As was noted after Lemma \ref{lem:3-spec}, we have $\varclass{A} = \varclass{A^{(i,j,k)}_{t}}$ with some permutation $(i,j,k)$ if and only if $\#\sigma(A)=3$ and
\begin{align*}
    \{B\in\Bsa(H)\colon B\preceq A, \#\sigma(B)\leq 2\}
    = \bigcup_{\substack{0\leq r\leq t \\ l\in\{0,1,2\}}} \varclass{rP_l}.
\end{align*}
Note that $\#\sigma(\Phi(A))=3$ and that the $\Phi$-image of the above set is 
\begin{align*}
    \{B\in\Bsa(H)\colon B\preceq \Phi(A), \#\sigma(B)\leq 2\}
    = \bigcup_{\substack{0\leq r\leq \xi(t) \\ l\in\{0,1,2\}}} \varclass{rP_l}.
\end{align*}
Therefore, we conclude that $\varclass{A} = \varclass{A^{(i,j,k)}_{t}}$ with some permutation $(i,j,k)$ if and only if $\varclass{\Phi(A)} = \varclass{A^{(i',j',k')}_{\xi(t)}}$ with some permutation $(i',j',k')$.
In particular,
$$
\Phi\left( \varclass{A^{(0,1,2)}_{t}}\cup\varclass{A^{(2,0,1)}_{t}}\cup\varclass{A^{(1,2,0)}_{t}} \right) = \varclass{A^{(0,1,2)}_{\xi(t)}}\cup\varclass{A^{(2,0,1)}_{\xi(t)}}\cup\varclass{A^{(1,2,0)}_{\xi(t)}}.
$$
Note that comparing the distance between the spectrum points gives the following chain of equivalences for any $s>0$: we have $2t\leq s$ if and only if
$$
    A^{(i_1,j_1,k_1)}_{t}\preceq A^{(i_2,j_2,k_2)}_{s}
$$
is satisfied for all permutations $(i_1,j_1,k_1), (i_2,j_2,k_2)$, which holds exactly when 
$$
    A^{(i_1,j_1,k_1)}_{\xi(t)}\preceq A^{(i_2,j_2,k_2)}_{\xi(s)}
$$
is fulfilled for all permutations $(i_1,j_1,k_1), (i_2,j_2,k_2)$, and this holds if and only if $2\xi(t)\leq \xi(s)$. This immediately implies $\xi(2t) = 2\xi(t)$ for all $t>0$, which by an elementary algebraic and continuity argument gives $\xi(t) = \alpha t$ with some positive number $\alpha$.

Note that $\Phi$ satisfies our conditions if and only if $\frac{1}{\alpha}\Phi$ does, therefore without loss of generality we may, and from now on will, assume that 
\begin{equation}\label{eq:2-element-form}
    \Phi(\varclass{tP}) = \varclass{tP} \qquad (t>0, P\in\Proj(H)\setminus\{0,I\}).
\end{equation}
We shall see that this assumption implies that $\Phi\left(\varclass{A}\right) = \varclass{A}$ for all $A$.

\smallskip

STEP 5: \emph{We show that $\Phi\left(\varclass{A}\right) = \varclass{A}$ holds for all $A$ with $\#\sigma(A)=3$.}

We continue to use the notation of the previous step. Further to the projections $P_0,P_1,P_2$, let us fix $t_0>0$. By Lemma \ref{lem:3-spec}, \eqref{eq:spec-card-pres} and \eqref{eq:2-element-form}, we have
\begin{equation}\label{eq:t<t0}
    \Phi\left(\varclass{A_{t_0,t}^{(0,1,2)}}\right) = \varclass{A_{t_0,t}^{(0,1,2)}} \;\;\text{or}\;\; \Phi\left(\varclass{A_{t_0,t}^{(0,1,2)}}\right) = \varclass{A_{t_0,t}^{(0,2,1)}} \quad (0<t<t_0),
\end{equation}
and
\begin{equation}\label{eq:t>t0}
    \Phi\left(\varclass{A_{t_0,t}^{(0,1,2)}}\right) = \varclass{A_{t_0,t}^{(0,1,2)}}
    \;\;\text{or}\;\;
    \Phi\left(\varclass{A_{t_0,t}^{(0,1,2)}}\right) =\varclass{A_{t_0,t}^{(1,0,2)}} \quad (t>t_0).
\end{equation}

First, we claim that in \eqref{eq:t<t0}, we either have the first option for all $0<t<t_0$, or the second for all $0<t<t_0$. Suppose this is not so, then there exist $0<t'<t<t_0$ such that for $t'$ the first option holds and for $t$ the second, or the other way around. We only show that the former case is impossible, as the latter is very similar to handle. Since $A_{t_0,t}^{(0,1,2)}\succeq A_{t_0,t'}^{(0,1,2)}$, we must have $A_{t_0,t}^{(0,2,1)}\succeq A_{t_0,t'}^{(0,1,2)}$. However a consideration of the distances between the eigenvalues corresponding to $P_0$ and $P_2$ gives the contradiction $t_0 \geq t_0+t'$. This proves our claim.

Second, we show that in \eqref{eq:t>t0}, we either have the first option for all $t>t_0$, or the second for all $t>t_0$. If this were not so, then it would be possible to find $t$ and $t'$ with $t_0+t>t'>t>t_0$ such that for $t'$ the first option holds and for $t$ the second, or the other way around. The former case would imply $A_{t_0,t}^{(1,0,2)}\preceq A_{t_0,t'}^{(0,1,2)}$, since $A_{t_0,t}^{(0,1,2)}\preceq A_{t_0,t'}^{(0,1,2)}$. However, comparing the distances between the eigenvalues corresponding to $P_1$ and $P_2$ gives $t_0+t \leq t'$, which cannot be. The latter case implies a similar contradiction.

Third, we prove that 
\begin{equation}\label{eq:allt}
    \Phi\left(\varclass{A_{t_0,t}^{(0,1,2)}}\right) = \varclass{A_{t_0,t}^{(0,1,2)}} \quad (t>0),
\end{equation}
which will finish this step. Choose $t$ and $t'$ such that $0<t'<t_0<t<t_0+t'$ holds. The inequality $A_{t_0,t'}^{(0,1,2)}\preceq A_{t_0,t}^{(0,1,2)}$ implies one of the following, depending on which cases hold in \eqref{eq:t<t0} and \eqref{eq:t>t0}:
\begin{itemize}
    \item $A_{t_0,t'}^{(0,2,1)}\preceq A_{t_0,t}^{(1,0,2)}$, which gives $t_0+t' \leq t_0$, that is impossible,
    \item $A_{t_0,t'}^{(0,2,1)}\preceq A_{t_0,t}^{(0,1,2)}$, which gives $t_0+t' \leq t_0$, that is also impossible,
    \item $A_{t_0,t'}^{(0,1,2)}\preceq A_{t_0,t}^{(1,0,2)}$, which gives $t_0+t' \leq t$, that is again impossible,
    \item $A_{t_0,t'}^{(0,1,2)}\preceq A_{t_0,t}^{(0,1,2)}$, which therefore must be the case.
\end{itemize}
By the previous two paragraphs, this proves that \eqref{eq:allt} must hold for all $t\neq t_0$. We saw in the previous step that for $t=t_0$ we have either $\Phi\left(\varclass{A_{t_0}^{(0,1,2)}}\right)=\varclass{A_{t_0}^{(0,1,2)}}$, or $\Phi\left(\varclass{A_{t_0}^{(0,1,2)}}\right)=\varclass{A_{t_0}^{(2,0,1)}}$, or $\Phi\left(\varclass{A_{t_0}^{(0,1,2)}}\right)=\varclass{A_{t_0}^{(1,2,0)}}$.
However, since for all $t_0<t<2t_0$ we have 
$$
A_{t_0}^{(0,1,2)} \preceq A_{t_0,t}^{(0,1,2)},\;\;
A_{t_0}^{(2,0,1)} \not\preceq A_{t_0,t}^{(0,1,2)},\;\;
A_{t_0}^{(1,2,0)} \not\preceq A_{t_0,t}^{(0,1,2)},
$$
and $\Phi$ preserves $\preceq$, we obtain $\Phi\left(\varclass{A_{t_0}^{(0,1,2)}}\right)=\varclass{A_{t_0}^{(0,1,2)}}$, that completes the proof of \eqref{eq:allt}.

\smallskip

STEP 6: \emph{We show $\Phi\left(\varclass{A}\right) = \varclass{A}$ for all $A$.}

Lemma \ref{lem:for-ind} implies $\Phi\left(\varclass{A}\right) = \varclass{A}$ for all $A$ with $\#\sigma(A)<\infty$. This in turn completes the proof in the finite dimensional case. For the infinite dimensional case we use a straightforward continuity argument to finish the proof. Finally, note that throughout the proof we transformed our original $\Phi$ twice (see at the end of STEPs 2 and 4). If we transform our obtained map back, one arrives at the form \eqref{eq:pres-form}.
\end{proof}

\section*{Acknowledgement}
Geh\'er was supported by the Leverhulme Trust Early Career Fellowship (ECF-2018-125), and also by the Hungarian National Research, Development and Innovation
Office - NKFIH (grant no. K134944).


\begin{thebibliography}{99}

\bibitem{AS-order-abstract}
S.~Artstein-Avidan, B.A.~Slomka,
Order isomorphisms in cones and a characterization of duality for ellipsoids, Selecta Math. (N.S.) 18 (2012), no. 2, 391--415. 

\bibitem{Lip-ext}
S.~Banach, Introduction to the theory of real functions, Warszawa-Wrocław, 1951, pp. 121--122

\bibitem{Bog1}
V.~Bogachev,
Measure theory (volume I),
Springer, 2007.

\bibitem{GeSeCMP}
G.P.~Geh\'er, P.~\v{S}emrl, 
Coexistency on Hilbert space effect algebras and a characterisation of its symmetry transformations, 
Comm. Math. Phys. 379 (2020), no. 3, 1077--1112.

\bibitem{IRS-order}
A.N.~Iusem, D.~Reem, B.F.~Svaiter, Order preserving and order reversing operators on the class of convex functions in Banach spaces. J. Funct. Anal. 268 (2015), no. 1, 73--92.

\bibitem{MoBa}
L.~Moln\'ar, M.~Barczy,
Linear maps on the space of all bounded observables preserving maximal deviation,
J. Funct. Anal. 205 (2003), no. 2, 380--400. 

\bibitem{Molnar-order1}
L.~Moln\'ar,
Order-automorphisms of the set of bounded observables,
J. Math. Phys. 42 (2001), no. 12, 5904--5909.

\bibitem{Molnar-book}
L.~Moln\'ar, Selected preserver problems on algebraic structures of linear operators and on function spaces, Lect. Notes Math. 1895, Springer-Verlag, 2007.
 
\bibitem{Molnar-order2}
L.~Moln\'ar, Order automorphisms on positive definite operators and a few applications, Linear Algebra Appl. 434 (2011), no. 10, 2158--2169.

\bibitem{MoSe-comm}
L.~Moln\'ar, P.~\v{S}emrl, Nonlinear commutativity preserving maps on self-adjoint operators, Q. J. Math. 56 (2005), no. 4, 589--595.

\bibitem{MoS2}
L.~Moln\' ar, P.~\v Semrl,
Transformations of the unitary group on a Hilbert space, 
{J. Math. Anal. Appl.} {388} (2012),  1205--1217.

\bibitem{Mori-order1}
M.~Mori, Order isomorphisms of operator intervals in von Neumann algebras, Integral Equations Operator Theory 91 (2019), no. 2, Paper No. 11, 26 pp.

\bibitem{Se-comp}
P.~\v{S}emrl, Comparability preserving maps on bounded observables, Integral Equations Operator Theory 62 (2008), no. 3, 441--454.

\bibitem{Se-symm}
P.~\v{S}emrl, Symmetries on bounded observables: a unified approach based on adjacency preserving maps, Integral Equations Operator Theory 72 (2012), no. 1, 7--66.

\bibitem{Se-symm2}
P.~\v{S}emrl, Symmetries of Hilbert space effect algebras, J. Lond. Math. Soc. (2) 88 (2013), no. 2, 417--436.

\bibitem{Se-order1}
P.~\v{S}emrl, Order isomorphisms of operator intervals, Integral Equations Operator Theory 89 (2017), no. 1, 1--42.

\bibitem{Se-order2}
P.~\v{S}emrl, Order and spectrum preserving maps on positive operators, Canad. J. Math. 69 (2017), no. 6, 1422--1435.

\bibitem{Se-order3}
P.~\v{S}emrl, Groups of order automorphisms of operator intervals, Acta Sci. Math. (Szeged) 84 (2018), no. 1--2, 125--136.

\bibitem{W-order}
C.~Walsh, Order antimorphisms of finite-dimensional cones, Selecta Math. (N.S.) 26 (2020), no. 4, Paper No. 53, 15 pp.

\end{thebibliography}
\end{document}